\documentclass[11pt]{article}
\usepackage[utf8]{inputenc}
\usepackage[dvipsnames]{xcolor}
\usepackage{graphicx}
\usepackage{amsmath}
\usepackage{amsfonts}
\usepackage{amssymb}
\usepackage{amsthm}
\usepackage{thmtools}
\usepackage[margin=25mm]{geometry}

\usepackage{caption}
\usepackage{enumitem,calc}
\usepackage[longnamesfirst,numbers,sort&compress]{natbib}

\usepackage[unicode=true]{hyperref}
\hypersetup{
colorlinks,
linkcolor={black},
citecolor={black},
urlcolor={blue!60!black}
}

\usepackage[noabbrev,capitalise,nameinlink]{cleveref}
\crefname{lemma}{Lemma}{Lemmas}
\crefname{theorem}{Theorem}{Theorems}
\crefname{corollary}{Corollary}{Corollaries}
\crefname{proposition}{Proposition}{Propositions}
\crefname{figure}{Figure}{Figures}

\theoremstyle{plain}
\newtheorem{theorem}{Theorem}

\newtheorem{lemma}[theorem]{Lemma}

\newtheorem{observation}[theorem]{Observation}

\theoremstyle{definition}

\DeclareMathOperator{\tw}{tw}
\DeclareMathOperator{\rtw}{rtw}

\DeclareMathOperator{\polylog}{polylog}
\newcommand{\scr}[1]{\mathcal{#1}}
\newcommand{\ceil}[1]{\lceil#1\rceil}
\newcommand{\floor}[1]{\lfloor#1\rfloor}
\newcommand{\ds}[1]{\mathbb{#1}}
\newcommand{\eps}{\epsilon}
\newcommand{\UWT}[2]{#1\langle #2\rangle}

\newcommand{\Apex}[2]{#1^{+#2}}

\newcommand{\defn}[1]{\textcolor{Maroon}{\emph{#1}}}

\renewcommand{\geq}{\geqslant}
\renewcommand{\leq}{\leqslant}

\begin{document}

\title{\bf{Proper Minor-Closed Classes in Blowups of Fans}}
\author{Marc Distel\,\footnotemark[2]}
\date{\today}

\footnotetext[2]{School of Mathematics, Monash University, Melbourne, Australia (\texttt{marc.distel@monash.edu}). Research supported by Australian Government Research Training Program Scholarship.}

\maketitle

\begin{abstract}
    Dujmovi\'{c} et al. [arXiv:2407.05936] showed that any $n$-vertex planar graph is contained in a $O(\sqrt{n}\log^2(n))$-blowup of a fan, and asked if the same holds for any $n$-vertex $K_t$-minor-free graph. We answer this in the positive, showing that such a graph is contained in a $O_t(\sqrt{n}\log^2(n))$-blowup of a fan.
\end{abstract}

\section{Introduction}

    The strong product of graphs $G$ and $H$ is the graph with vertex set $V(G)\times V(H)$ and edges between two distinct pairs of vertices if they are adjacent or equal in each coordinate. There has been substantial recent interest in describing complex graphs as subgraphs of strong products of simple graphs \citep{rtwltwMCC,Distel2022Surfaces,Illingworth2022,Distel2024,Dvorak2023,Dujmovic2024,Ueckerdt2022}. In particular, several papers have shown that various complex graphs are contained in blowups of simple graphs \citep{Illingworth2022,Distel2024,Dvorak2023,Dujmovic2024}, where the \defn{$b$-blowup} of a graph $H$ is the graph $H\boxtimes K_b$, and a graph $G$ is \defn{contained} in a graph $G'$ if $G$ is isomorphic to a subgraph of $G'$. We seek to make $H$ as simple as possible while keeping $b$ relatively small. Specifically, when studying $n$-vertex graphs $G$ with treewidth $O(\sqrt{n})$, we want $H$ to have treewidth at most some absolute constant $k$, and we want $b\in \Tilde{O}(\sqrt{n})$, where $\Tilde{O}$ hides any $\polylog(n)$ terms. Note that this would imply $G$ has treewidth at most $(k+1)b\in \Tilde{O}(\sqrt{n})$, thus these blowup results are qualitative strengthenings of classical treewidth results by \citet{Lipton1979} and \citet{Alon1990}.

    In this paper, we focus on proper minor-closed classes, which have treewidth $O(\sqrt{n})$~\citep{Alon1990}. See \citet{Dvorak2023} for a more general study about blowup structure of graph classes that admit $O(n^{1-\eps})$-balanced separators. 
    
    There has been substantial recent progress on this topic. \citet{Illingworth2022} showed that every $n$-vertex $K_h$-minor-free graph $G$ is contained in a $O_h(\sqrt{n})$-blowup of a graph of treewidth $h-2$, where $O_h$ hides multiplicative dependence on $h$. \citet{Distel2024} extended this, decreasing the treewidth to $4$ while remaining a $O_h(\sqrt{n})$-blowup, although we remark that the dependence on $h$ in the $O_h$ is much worse than in \citet{Illingworth2022}'s result.
    
    \citet{Distel2024} also showed that planar graphs are contained in a $O(\sqrt{n})$-blowup of a graph of treewidth $2$. They also asked if this could be improved to a $O(\sqrt{n})$-blowup of a graph of bounded pathwidth. \citet{Dujmovic2024} answered this up to a $\polylog(n)$ factor, showing that planar graphs are contained in a $\Tilde{O}(\sqrt{n})$-blowup of a fan, where a \defn{fan} is a graph obtained from a path by adding a vertex, called the \defn{centre}, adjacent to every other vertex. We remark that fans have pathwidth $2$.
    
    \begin{theorem}[\citet{Dujmovic2024}]
        \label{planarNonFlex}
        For each $n\in\mathbb{N}$, there is a $O(\sqrt{n}\log^2(n))$-blowup of a fan that contains every $n$-vertex planar graph.
    \end{theorem}

    \citet{Dujmovic2024} asked if this could be extended to proper minor-closed classes. Specifically, is every $K_h$-minor-free graph contained in a $\Tilde{O}_h(\sqrt{n})$-blowup of a fan? We answer this in the positive.
    
    \begin{theorem}
        \label{main}
        For each $h,n\in\mathbb{N}$, there is a $O_h(\sqrt{n}\log^2(n))$-blowup of a fan that contains every $n$-vertex $K_h$-minor-free graph.
    \end{theorem}

    Up to a $\polylog(n)$ factor, this answers a question of \citet{Distel2024}. They asked if all $K_h$-minor-free graphs are contained in a $O_h(\sqrt{n})$-blowup of a graph of bounded pathwidth. Since fans have pathwidth $2$, \cref{main} gives a positive answer, up to a $\polylog(n)$ factor.

\subsection{Definitions and Key Tools}

    Let $\ds{R}^+$ denote the set of all strictly positive real numbers, and let $\ds{N}:=\{1,2,\dots\}$.

    We consider a graph $G$ to be finite, simple, and undirected, with vertex set $V(G)$ and edge set $E(G)$.

    A \defn{rooted tree} is a tree $T$ with some fixed vertex $r$ called the \defn{root}. The \defn{induced root} of a subtree $T'$ of $T$ is the vertex of $T'$ closest to $r$ in $T$.

    For an integer $d\geq 1$, the \defn{$d$-th power} of a graph $G$, \defn{$G^d$}, is the graph with vertex set $V(G)$ and edges between two vertices if they are at distance at most $d$ in $G$.

    A \defn{class} of graphs is a collection of graphs closed under isomorphism. A class is \defn{monotone} if it is closed under taking subgraphs. Given a class of graphs $\scr{G}$, we use \defn{$\widehat{\scr{G}}$} to denote the class of graphs $G$ such that there exists a set $Z\subseteq V(G)$ of degree 1 vertices in $G$ such that $G-Z\in \scr{G}$. For a class of graphs $\scr{G}$ and integer $a\geq 0$, we use \defn{$\Apex{\scr{G}}{a}$} to denote the set of graphs $G$ such that there exists $X\subseteq V(G)$ with $|X|\leq a$ such that $G-X\in \scr{G}$. We call $X$ the \defn{apices} of $G$. Observe that if $\scr{G}$ is closed under adding isolated vertices, then $\widehat{\Apex{\scr{G}}{a}}\subseteq \Apex{\widehat{\scr{G}}}{a}$.

    A \defn{star} is a tree with at least one vertex, called the \defn{centre}, that is adjacent to every other vertex. Note that if a star contains at least three vertices, then the centre is unique.

    A graph $H$ is a \defn{minor} of a graph $G$, denoted \defn{$H\leq G$}, if a graph isomorphic to $H$ can be obtained from a subgraph of $G$ by performing any number of contractions; otherwise, $G$ is \defn{$H$-minor-free}. A class of graphs $\scr{G}$ is \defn{$H$-minor-free} if every graph in $\scr{G}$ is $H$-minor-free. $\scr{G}$ is \defn{minor-closed} if, for every $G\in\scr{G}$, every minor of $G$ is in $\scr{G}$. $\scr{G}$ is \defn{proper} if it is not the class of all graphs. Note that if $\scr{G}$ is a proper minor-closed class, then $\scr{G}$ is $K_h$-minor-free for some $h\in \ds{N}$.

    A function $f:X\mapsto Y$ with $X,Y\subseteq \ds{R}^+$ is \defn{increasing} if for all $n,m\in X$ with $n<m$, $f(n)<f(m)$, and is \defn{superadditive} if $f(n)+f(m)\leq f(n+m)$ for all $n,m\in X$. Note that if a superadditive function has strictly positive outputs ($Y\subseteq \ds{R}^+$) then it is also increasing. For a constant $c\in \ds{R}^+$, we use the notation \defn{$cf$} to denote the function $n\mapsto cf(n)$, the notation \defn{$f/c$} to denote the function $n\mapsto f(n)/c$, and the notation \defn{$f+c$} to denote the function $n\mapsto f(n)+c$.

    A \defn{partition} $\scr{P}$ of a graph $G$ is a collection of pairwise disjoint subsets of $V(G)$ such that $\bigcup_{P\in \scr{P}} P=V(G)$. The subsets $P\in \scr{P}$ are called the \defn{parts} of $\scr{P}$. Note that parts may be empty. The \defn{width} of $\scr{P}$ is $\max_{P\in \scr{P}}|P|$. The \defn{quotient} of $\scr{P}$ in $G$, denoted \defn{$G/\scr{P}$}, is the graph with vertex set $\scr{P}$ and edges between distinct parts if the subgraphs of $G$ they induce are adjacent in $G$. For a graph $H$, we say that $\scr{P}$ is a \defn{$H$-partition} if $G/\scr{P}$ is contained in $H$. In particular, if $\scr{P}$ is an $F$-partition for a fan $F$, then we say that $\scr{P}$ is a \defn{fan-partition}, and if $\scr{P}$ is an $P$-partition for a path $P$, then we say that $\scr{P}$ is a \defn{path-partition}. For a fan-partition $\scr{P}$ of a graph $G$, we call the part $P$ that corresponds to the centre $v$ of the fan $F$ the \defn{central part}. Note that we may assume that the central part always exists, by taking it to be empty if need be.

    For an integer $w \geq 1$, a graph $G$ admits a $H$-partition of width at most $w$ if and only if it is contained in the $w$-blowup of $H$. We work with this as the definition from now on.

    \cref{main,planarNonFlex} are stated in terms of a universal fan whose blowup contains every $n$-vertex graph in the corresponding class. However, the following observation allows us to reduce to finding a fan-partition for each graph in the class individually.
    \begin{observation}
        \label{unvToInd}
        For any $n,w\in \ds{N}$ and any class of $n$-vertex graphs $\scr{G}$, there exists a fan $F$ such that each $G\in \scr{G}$ is contained in the $w$-blowup of $F$ if and only if for each $G\in \scr{G}$, there exists a fan $F_G$ such that $G$ is contained in the $w$-blowup of $F_G$.
    \end{observation}
    \begin{proof}
        The forward direction is immediate. For the backward direction, by contracting adjacent parts in $F_G$ whenever they both have size at most $w/2$, we can assume that at least half the parts have size at least $w/2$. Since the class of subgraphs of fans is minor-closed, the resulting partition is still a $F_G'$-partition for some fan $F_G'$ with at most $4n/w$ vertices. Since every fan on $m$ vertices is contained in every fan on $m'\geq m$ vertices, it suffices to take $F$ to be a fan on $\floor{4n/w}$ vertices.
    \end{proof}
    
    Additionally, a graph $G$ admits a fan-partition of width at most $w$ if and only if there exists $X\subseteq V(G)$ with $|X|\leq w$ such that $G-X$ admits a path-partition of width at most $w$. We now expand on this point.

    Let $\scr{P}$ be a partition of a graph $G$. For real numbers $k,w\geq 0$, we say that $\scr{P}$ is a \defn{$(k,w)$-fan-partition} of $G$ if there exists $P\in \scr{P}$ with $|P|\leq k$ such that $\scr{P}\setminus \{P\}$ is a path-partition of $G-P$ of width at most $w$. Note that $\scr{P}$ is a fan-partition, and that $\scr{P}$ has width at most $\max(k,w)$. We say that $G$ admits \defn{$(k,w)$-flexible fan-partitions} if for all $d\in \ds{R}^+$ with $d\geq 1$, $G$ admits a $(k/d,wd)$-fan-partition. Note that provided $w\geq 1$, the existence of a $(k/d,wd)$-fan-partition is trivial when $d>|V(G)|$. So it suffices to consider $d\in \ds{R}^+$ with $1\leq d\leq |V(G)|$.
    
    Given functions $f,g:\ds{N}\mapsto \ds{R}$, we say that a class $\scr{G}$ admits \defn{$(f,g)$-fan-partitions} if for each $G\in \scr{G}$, $G$ admits a $(f(|V(G)|),g(|V(G)|))$-fan-partition. We say that $\scr{G}$ admits \defn{$(f,g)$-flexible fan-partitions} if for every $d\in \ds{R}^+$ with $d\geq 1$, $\scr{G}$ admits $(f/d,dg)$-fan-partitions. Note that by the above observation, provided $g(n)\geq 1$ for all $n\in \ds{N}$, it suffices to check $d\in \ds{R}^+$ with $1\leq d\leq |V(G)|$ for each graph $G$.

    A \defn{tree-decomposition} of a graph $G$ is a collection $(B_t\subseteq V(G):t\in V(T))$ with $T$ a tree such that (a) for each $v\in V(G)$, the subgraph of $T$ induced by the vertices $t\in V(T)$ with $v\in B_t$ is a nonempty subtree; and
    (b) for each $uv\in E(G)$, there exists $t\in V(T)$ such that $u,v\in B_t$. The sets $B_t$, $t\in V(T)$, are called \defn{bags}. If $T$ is a star $S$, then $(B_s:s\in V(S))$ is a \defn{star-decomposition} of $G$. If $T$ is a path $P$, then $(B_p:p\in V(P))$ is a \defn{path-decomposition}. The \defn{width} of a tree-decomposition is $\max_{t\in V(T)}(|B_t|-1)$. The \defn{adhesion} of a tree-decomposition is $\max_{tt'\in E(T)}|B_t\cap B_{t'}|$. The \defn{torso} of $G$ at a vertex $t\in V(T)$ (with respect to $(B_t:t\in V(T))$), denoted \defn{$\UWT{G}{B_t}$}, is the graph obtained from $G[B_t]$ by adding an edge $uv$ whenever there exists $tt'\in E(T)$ such that $u,v\in B_t\cap B_{t'}$, provided $uv$ does not already exist in $G[B_t]$.

    The \defn{treewidth} of a graph $G$, \defn{$\tw(G)$}, is the minimum width of a tree-decomposition of $G$. The \defn{pathwidth} of $G$ is the minimum width of a path-decomposition of $G$.

    The \defn{row treewidth} of a graph $G$, \defn{$\rtw(G)$}, is the minimum $b$ such that there exists a graph $H$ with $\tw(H)\leq b$ and a path $P$ such that $G$ is contained in $H \boxtimes P$. Note that row treewidth does not increase under taking subgraphs or adding isolated vertices. The \defn{row treewidth} of a class $\scr{G}$, \defn{$\rtw(\scr{G})$}, is the maximum row treewidth of a graph in $\scr{G}$, or $\infty$ if the maximum does not exist.

    \citet{Dujmovic2024} actually showed a stronger version of \cref{planarNonFlex} in terms of flexible-fan-partitions and row treewidth.

    \begin{theorem}[\citet{Dujmovic2024}, Theorem~49]
        \label{rtwFlex}
        Every $n$-vertex graph $G$ with row treewidth at most $b$ admits $(f,g)$-flexible fan-partitions for some $f\in O(bn\log(n))$ and $g\in O(\log^3(n))$.
    \end{theorem}

    We remark that Theorem~49 of \citet{Dujmovic2024} only shows the $1\leq d\leq |V(G)|$ case, however the $d>|V(G)|$ follows trivially by previous observations since we may assume $f(n)\geq 1$ for all $n\in \ds{N}$.
    
    Note that \cref{planarNonFlex} follows directly from \cref{rtwFlex} by taking $d:=n/\log(n)$ in the definition of flexible fan-partitions, since planar graphs have row treewidth at most $8$ \citep{rtwltwMCC}, improved to $6$ by \citet{Ueckerdt2022}.

    \cref{rtwFlex} will be one of our key tools for proving \cref{main}. The other main tool is the `Graph Minor Product Structure Theorem' of \citet{rtwltwMCC}, which gives a structural description of $K_h$-minor free graphs.

    \begin{theorem}[Graph Minor Product Structure Theorem]
        \label{GMST}
        For every integer $h\geq 2$, there exists integers $b,k\geq 1$ such that every $K_h$-minor-free graph $G$ admits a tree-decomposition $(B_t:t\in V(T))$ of adhesion at most $k$ such that for each $t\in V(T)$, there exists $X_t\subseteq B_t$ with $|X_t|\leq \max(h-5,0)$ such that $\UWT{G}{B_t}-X_t$ has row treewidth at most $b$.
    \end{theorem}

    We remark that the Graph Minor Product Structure Theorem employs the `Graph Minor Structure Theorem' of \citet{Robertson2003}.
    
\section{Proofs}
    Before we get into the proof of \cref{main}, we give an overview of the idea. We recall that, by \cref{unvToInd}, to prove the existence of a fan-blowup containing every $n$-vertex graph $G$, it suffices only to show, for each graph $G$, the existence of a fan-blowup containing $G$.
    
    The proof consists of two steps. The first is `splitting up the tree-decomposition', and the second is `processing the smaller sections'. Step 1 is primarily handled by \cref{mainLem}, Step 2 by \cref{starDecomp}, and all other results assist in one of these two steps. We now explain the ideas behind these steps.

    Starting with a $K_h$-minor-free graph $G$, let $b,k$ be from \cref{GMST}, let $\scr{H}$ be the class of all graphs with row treewidth at most $b$, and let $a:=\max(h-5,0)$. By \cref{GMST}, there is a tree-decomposition $(B_t:t\in V(T))$ of $G$ of adhesion at most $k$ such that each torso is in $\Apex{\scr{H}}{a}$. Our goal is to find a small number of vertices whose removal breaks the graph, and the tree-decomposition, into smaller chunks. To do this, we use the same method as \citet{Distel2024}. We can find a small set $Z\subseteq V(T)$ whose deletion splits $T$ into subtrees $T'$ such that the total number of vertices contained in $\bigcup_{t\in V(T')}B_t$ is small; see \cref{treeDeletions} and \cref{mainLem}. 
    
    We then consider $T$ to be rooted at a vertex $r$. For each $z\in Z\setminus \{r\}$, delete the at most $k$ vertices in $B_z$ that are in the bag of $z$'s parent. The total number of vertices deleted is small, but it has the effect of breaking the graph and tree-decomposition into manageable chunks. Specifically, each chunk has a star-decomposition $(B'_s:s\in V(S))$ of adhesion at most $k$ where, if $r'$ is the centre of $S$, then the torso at $r'$ is in $\Apex{\scr{H}}{a}$, and each bag $B'_s$, $s\neq r'$, is small.
    
    We then process these chunks separately, finding flexible-fan-partitions for each. Using this flexibility, we can guarantee that the total number of vertices in central parts across all chunks is still small, while maintaining that the width of the corresponding path-partition of each chunk is also small, with respect to the size of the total graph. We then merge all the central parts into one single part, and add all the vertices of $G$ that we deleted into this part. This part will be our new central part, and will still have small size. Noting that a graph admits a path-partition of width $w$ if and only if each connected component does, this yields the desired fan-partition.

    We now sketch Step 2, processing the smaller chunks. As the central torso is in $\Apex{\scr{H}}{a}$, the majority of the work is done by \cref{rtwFlex}, as we can just put the apices in the central part; see \cref{addApex}. However, we still need to attach the vertices in the smaller bags $B'_s$, $s\neq r$. For this, we create an auxiliary graph $G'$, obtained from the central torso by attaching degree 1 vertices. Specifically, for each vertex $v$ in some bag $B'_s$, $s\neq r$, and each vertex $x$ in $B'_s\cap B'_r$, we add an auxiliary vertex $v_x$ adjacent to $x$. The resulting graph is in $\widehat{\Apex{\scr{H}}{a}}$. However, $\widehat{\Apex{\scr{H}}{a}}=\Apex{\scr{H}}{a}$ as $\widehat{\scr{H}}=\scr{H}$ (see \cref{rtwExt}) and since $\scr{H}$ is closed under isolated vertices. Thus, we can still process $G'$ using \cref{rtwFlex,addApex}. We obtain a fan-partition of $G'$. Call the central part $X$.
    
    We find that for each $B'_s$, $s\neq r$, one of two things must happen. One possibility is that $B'_s\cap B'_r\subseteq X$, in which case $B'_s$ can exist as an isolated part in the path-partition. The second possibility is that there exists some $x\in (B'_s\cap B'_r)\setminus X$. In this case, for each vertex $v\in B'_s\setminus B'_r$, one of two outcomes must have occurred. Either $v_x\in X$, in which case $v$ is added to our central part, or $v_x$ is at distance at most $2$ in $G'-X$ of every vertex in $(B'_s\cap B'_r)\setminus X$ and every vertex $v'_x$ with $v'\in B'_s\setminus B'_r$ and $v'_x\notin X$. In particular, these vertices become adjacent in $(G'-X)^2$. As one can still find a path-partition of small width even after taking a power of the graph, see \cref{PPpowers}, we can convert our path-partition to be a path-partition of $(G'-X)^2$. Hence, all these vertices belong to adjacent parts. Thus, converting this partition naturally to a partition of $G$, we find that this new partition is a path-partition, as all possible adjacencies were already handled by $(G'-X)^2$.

    We now start the formal proof of \cref{main}.

    \begin{lemma}
        \label{rtwExt}
        For every class of graphs $\scr{G}$, $\rtw(\widehat{\scr{G}})\leq \max(\rtw(\scr{G}),1)$.
    \end{lemma}

    \begin{proof}
        Let $b:=\max(\rtw(\scr{G}),1)$. Fix $G\in \widehat{\scr{G}}$. So there exists a set $Z$ of degree 1 vertices in $G$ such that $G-Z\in \scr{G}$. Thus, there exists a graph $H$ of treewidth at most $b$ and a path $P$ such that $G-Z$ is contained in $H\boxtimes P$. Let $E_z$ be the set of edges of $G$ with an endpoint in $Z$. Note that we can consider each edge in $E_z$ to be of the form $(h,p)z$ with $h\in V(H)$, $p\in V(P)$, $z\in Z$.
        
        Let $H'$ be the graph with vertex set $V(H)\cup Z$ and edge set $E(H)\cup \{hz:(h,p)z\in E_z\}$. Observe that $H'$ is obtained from $H$ by adding some number of degree 1 vertices. Thus, $\tw(H')\leq \max(\tw(H),1)\leq b$. Next, notice that for each $(h,p)z\in E_z$, we have $hz\in E(H')$ and thus $(h,p)(z,p)\in E(H'\boxtimes P)$. Therefore, we find that $G$ is isomorphic to a subgraph of $H'\boxtimes P$, via the mapping $(h,p)\mapsto (h,p)$ for $v=(h,p)\in V(G-Z)$, and $z\mapsto (z,p)$ for the unique $(h,p)\in V(G-Z)$ such that $(h,p)z\in E_z$. This gives the desired result.
    \end{proof}

    \begin{lemma}
        \label{addApex}
        Let $f,g:\ds{N}\mapsto \ds{R}^+$ be increasing functions, let $a\in \ds{N}\cup \{0\}$, and let $\scr{G}$ be a class of graphs that admit $(f,g)$-fan-partitions. Then $\Apex{\scr{G}}{a}$ admits $(f+a,g)$-fan-partitions.
    \end{lemma}

    \begin{proof}
        Fix $G\in \Apex{\scr{G}}{a}$, and let $n:=|V(G)|$. Let $X_1$ be the set of at most $a$ apices of $G$. So $G-X_1\in \scr{G}$. Let $\scr{P}$ be a $(f,g)$-fan-partition of $G-X_1$ with central part $X_2$. So $|X_2|\leq f(|V(G-X_1)|)\leq f(n)$, and $\scr{P}\setminus \{X_2\}$ is a path-partition of $G-X_1-X_2$ of width at most $g(|V(G-X_1)|)\leq g(n)$. Here, we use that $f$ and $g$ are increasing. 
        
        Let $X':=X_1\cup X_2$, note that $|X'|\leq f(n)+a$. It follows that $(\scr{P}\setminus \{X_2\})\cup \{X'\}$ is the desired fan-partition of $G$.
    \end{proof}

    \begin{lemma}
        \label{PPpowers}
        For each $d\in \ds{N}$ and every $w\in \ds{R}^+$, if a graph $G$ admits a path-partition of width at most $w$, then $G^d$ admits a path-partition of width at most $wd$.
    \end{lemma}

    \begin{proof}
        Let $\scr{P}=\{P_1,P_2,\dots,P_m\}$ be the path-partition of $G$, with the parts ordered by their position on the path. Set $P_i:=\emptyset$ for $i>m$ and $i<1$. Let $\scr{P}':=\{\bigcup_{j=1}^d P_{id+j}:i\in \{0,\dots,\ceil{\frac{m}{d}}-1\}\}$. Observe that for each $k\in \{1,\dots,m\}$, $k\in \{id+1,\dots,(i+1)d\}$ for exactly one $i\in \{0,\dots,\ceil{\frac{m}{d}}-1\}$. It follows that $\scr{P}'$ is also a partition of $G$ (and $G^d$). As each $P'\in \scr{P}'$ is a union of at most $d$ parts in $\scr{P}$, it is immediate that $\scr{P}'$ has width at most $wd$. So it remains only to show that $\scr{P}'$ is a path-partition of $G^d$.

        As $\scr{P}$ is a path-partition, it follows that for each $vw\in E(G^d)$, if $v\in P_k$, then $w\in P_j$ with $j\in \{k-d,\dots,k+d\}$. Thus, if $v\in P_i'$, then $k\in \{id+1,\dots,(i+1)d\}$, and $j\in \{(i-1)d+1,\dots,(i+2)d\}$. But $\bigcup_{\ell=(i-1)d+1}^{(i+2)d}P_{\ell}\subseteq P_{i-1}'\cup P_i'\cup P_{i+1}'$, so $w\in P_{i-1}'\cup P_i'\cup P_{i+1}'$. It follows that $\scr{P}'$ is a path-partition for $G^d$, as desired.
    \end{proof}

    We also use the following technical lemma, see \citet[Lemma~9]{Distel2024}, for a proof.

    \begin{lemma}
        \label{treeDeletions}
        For every integer $q\geq 0$ and $n\in \ds{R}^+$, every vertex-weighted tree $T$ with total weight at most $n$ has a set $Z$ of at most $q$ vertices such that each component of $T-Z$ has weight at most $\frac{n}{q+1}$. 
    \end{lemma}

    We now prove our first main lemma (Step 2 in the above sketch). As a guide to the reader, we will apply \cref{starDecomp,mainLem} with $f(n)\in O(n\log(n))$ and $g(n)\in O(\log^3(n))$.

    \begin{lemma}
        \label{starDecomp}
        Let $k,n\in \ds{N}$, let $w\in \ds{R}^+$, let $f,g:\ds{N}\mapsto \ds{R}^+$ be increasing functions, and let $\scr{H}$ be a class of graphs such that $\widehat{\scr{H}}$ admits $(f,g)$-fan-partitions. Let $G$ be an $n$-vertex graph that admits a star-decomposition $(B_s:s\in V(S))$ of adhesion at most $k$ such that, if $r$ is the centre of $S$, then:
        \begin{itemize}
            \item $\UWT{G}{B_r}\in \scr{H}$, and
            \item for each $s\in V(S)\setminus \{r\}$, $|B_s\setminus B_r|\leq w$.
        \end{itemize}
        Then there exists $X\subseteq V(G)$ with $|X|\leq f(kn)$ such that $G-X$ admits a path-partition of width at most $\max(2g(kn),w)$.
    \end{lemma}

    \begin{proof}
        For each $s\in V(S)\setminus \{r\}$, let $K_s:=B_s\cap B_r$ and $B_s':=B_s\setminus B_r$. Note that $|K_s|\leq k$ and $|B_s'|\leq w$. Let $G'$ be the graph obtained from $\UWT{G}{B_r}$ by adding, for each $s\in V(S)\setminus \{r\}$, each $v\in B_s'$, and each $x\in K_s$, a new vertex $v_x$ adjacent only to $x$. See \cref{FigDeg1}. Observe that $G'\in \widehat{\scr{H}}$, and that $|V(G')|\leq |B_r|+k|V(G)\setminus B_r|\leq k|V(G)|=kn$. For each $s\in V(S)\setminus \{r\}$ and each $v\in B_s'$, let $M_v:=\{v_x:x\in K_s\}$.
        \begin{figure}[!ht]
        \begin{center}
        \begin{minipage}{0.85\textwidth}
        \centering
        \includegraphics[width=\textwidth]{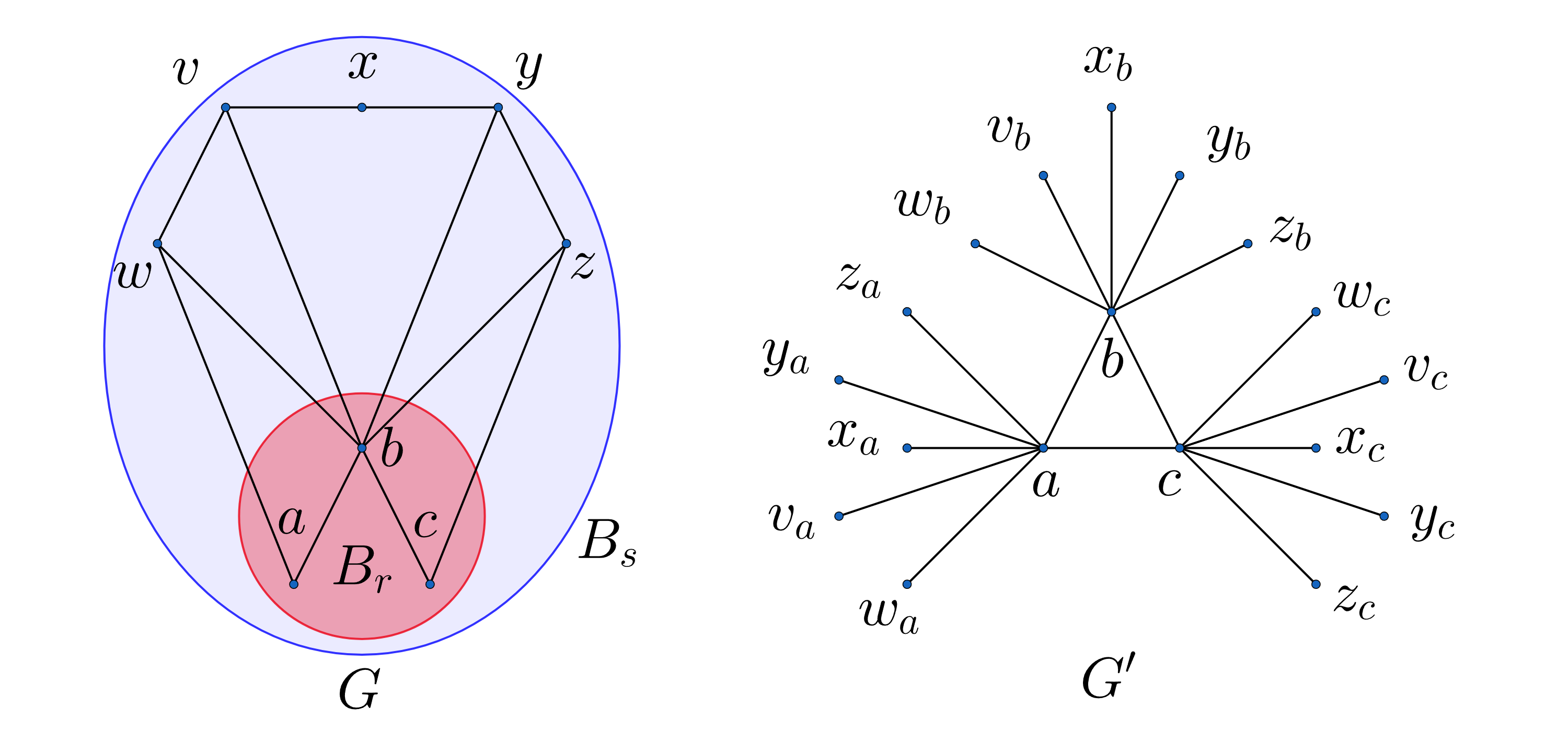}
        \caption{\label{FigDeg1}A diagram of how $G$ is turned into $G'$. The left graph is $G$, the right graph is $G'$. On the left, the vertices in the red section are $B_r$, where $r$ is the centre. The vertices in the blue section (all vertices) are $B_s$. Note that, on the right, each vertex in $B_s\setminus B_r=B_s'$ has been replaced by three new vertices, each adjacent to a different vertex of $B_r=\{a,b,c\}$, and that $B_r$ has become a clique. In this case, $M_v=\{v_a,v_b,v_c\}$ (and similar for $M_w,M_x,M_y,M_z$).}
        \end{minipage}
        \end{center}
        \end{figure}
        
        By definition of $\scr{H}$, there exists a set $X\subseteq V(G')$ with $|X|\leq f(|V(G')|)\leq f(kn)$ such that $G'-X$ admits a path-partition of width at most $g(|V(G')|)\leq g(kn)$. Here, we use that $f,g$ are increasing. By \cref{PPpowers}, $(G'-X)^2$ admits a path-partition $\scr{P}$ of width at most $2g(kn)$. Let $X'_1:=X\cap B_r$, let $X'_2:=\{v\in V(G)\setminus B_r:M_v\cap X\neq \emptyset\}$, and let $X':=X'_1\cup X'_2$. Observe that $|X'_2|\leq |X\setminus B_r|$, thus $|X'|\leq |X|\leq f(kn)$. Also, note that $X'\cap B_r=X\cap B_r=X'_1$.
        
        Let $U:=\{s\in V(S)\setminus \{r\}:K_s\setminus X=\emptyset\}$ and $W:=\{s\in V(S)\setminus \{r\}:K_s\setminus X\neq \emptyset\}$. Let $V_U:=(\bigcup_{s\in U}B_s')\setminus X'$ and $V_W:=(\bigcup_{s\in W}B_s')\setminus X'$. For each $s\in U$, define $P_s:=B_s'\setminus X'$. Note that $|P_s|\leq |B_s'|\leq w$. For each $s\in W$, pick some $x\in K_s\setminus X$, set $x_s:=x$ and, for each $v\in B_s'\setminus X'$, set $a_v:=v_x$. Note that as $v\notin X'\supseteq X'_2$, $M_v\cap X=\emptyset$. Hence, $a_v\in M_v$ is not in $X$.
        
        For $P\in \scr{P}$, define $P':=(P\cap B_r)\cup \{v\in V_W:a_v\in P\}$. Note that $|\{v\in V_W:a_v\in P\}|\leq |P\setminus B_r|$, thus $|P'|\leq |P|\leq 2g(kn)$. Let $\scr{P}':=\{P':P\in \scr{P}\}\cup \{P_s:s\in U\}$. We claim that $\scr{P}'$ is the desired path-partition of $G-X'$. We first show that it is a partition.
        
        Observe that $\{B_r\setminus X',V_U,V_W\}$ is a partition of $G-X'$. Also, note that $\{P_s:s\in U\}$ is a partition of $G[V_U]$, and that $\{P'\cap B_r:P\in \scr{P}\}=\{P\cap B_r:P\in \scr{P}\}$ is a partition for $G[B_r\setminus X]=G[B_r\setminus X']$. Next, for each $v\in V_W$, since $v\notin X'$, $a_v\notin X$. Thus, $a_v\in P$ for exactly one $P\in \scr{P}$, and $v\in P'$ for exactly one $P'\in \{Q':Q\in \scr{P}\}$. It follows that $\scr{P}'$ is a partition of $G$. 
        
        As $|P'|\leq 2g(kn)$ for each $P\in \scr{P}$ and $|P_s|\leq w$ for each $s\in U$, it is immediate that $\scr{P}'$ has width at most $\max(2g(kn),w)$. It remains only to show that $\scr{P'}$ is a path-partition.

        Let $H:=(G-X')/\scr{P}'$. Observe that for each $s\in U$, $P_s$ is an isolated vertex in $H$ as $K_s\subseteq X'$. So we only need to show that the subgraph of $H$ induced by $\{P':P\in \scr{P}\}$ is a subgraph of a path. 
        
        As $\scr{P}$ is a path-partition of $(G'-X)^2$, it suffices to show that distinct parts $P_1',P_2'$ are adjacent in $G-X'$ only if $P_1,P_2$ are adjacent in $(G'-X)^2$. Let $u\in P_1'$, $v\in P_2'$ be adjacent in $G$. Note that either $u\in P_1\cap B_r$, or $u\in V_W$ and $a_u\in P_1$. Likewise, either $v\in P_2\cap B_r$, or $v\in V_W$ and $a_v\in P_2$. Also, note that $u\notin X'$ and $v\notin X'$. 
        
        As $uv\in E(G)$, there exists $s\in V(S)$ such that $u,v\in B_s$. If $s=r$ then $u\in P_1$, $v\in P_2$, and $P_1,P_2$ are adjacent, as desired. So we may assume that $s\neq r$, and at least one of $u,v$ is not in $B_r$. 
        
        Recall that, for each $s\in W$, $x_s$ is a vertex in $K_s\setminus X$, and that for each $z\in B_s'$, $a_z$ is adjacent to $x_s$ in $G'$. Also, recall that for each $z\in B_s'\setminus X'$, $a_z\notin X$. Thus, if $u,v\in B_s'$, since $u\notin X'$ and $v\notin X'$, we have $a_u,a_v$ are at distance at most $2$ in $G'-X$ via the common neighbour $x_s$. As $a_u\in P_1$ and $a_v\in P_2$, we obtain that $P_1,P_2$ are adjacent in $(G'-X)^2$, as desired.

        The remaining case is if exactly one of $u,v$ is in $B_r$, w.l.o.g. say $u$. Thus, $u\in K_s$. Since $X'\cap B_r=X\cap B_r$ and since $u\notin X'$, we have $u\notin X$. Thus, $u$ is adjacent or equal to $x_s$ in $G'-X$ as $K_s$ is a clique in $G'[B_r]=\UWT{G}{B_r}$. By the same argument as before, $x_s$ is also adjacent to $a_v$. Thus, $x\in P_1$ is at distance at most 2 to $a_v\in P_2$ in $G'-X$. So $P_1$ and $P_2$ are adjacent in $(G'-X)^2$, as desired.
    \end{proof}
    
    \begin{lemma}
        \label{mainLem}
        Let $h\geq 2$, and let $k,b\geq 1$ be from \cref{GMST}. Let $f,g:\ds{N}\mapsto \ds{R}^+$ be functions with $f$ superadditive, $g$ increasing, and with $g(1)\geq 1$, such that the class of graphs of row treewidth at most $b$ admit $(f,g)$-flexible fan-partitions. Then the class of $K_h$-minor-free graphs admits $(f',g')$-flexible fan-partitions, where $f'(n):=f(kn)+(2\max(h-5,0)+k)n$ and $g'(n):=2g(kn)$.
    \end{lemma}

    \begin{proof}
        Note that because $f$ is superadditive and has strictly positive outputs, it is also increasing.
    
        Let $\scr{G}$ denote the class of graphs of row treewidth at most $b$, and let $a:=\max(h-5,0)$. By \cref{rtwExt}, $\widehat{\scr{G}}=\scr{G}$. Since $\scr{G}$ is closed under adding isolated vertices, $\widehat{\Apex{\scr{G}}{a}}\subseteq \Apex{\widehat{\scr{G}}}{a}=\Apex{\scr{G}}{a}$. Thus, $\widehat{\Apex{\scr{G}}{a}}=\Apex{\scr{G}}{a}$. 
        
        By definition of $f$ and $g$, for any $d\in \ds{R}^+$ with $d\geq 1$, $\scr{G}$ admits $(f/d,gd)$-fan-partitions. Thus, by \cref{addApex}, $\Apex{\scr{G}}{a}$ admits $(f/d+a,gd)$-fan-partitions. Note also that $\Apex{\scr{G}}{a}$ is monotone as $\scr{G}$ is monotone.
    
        Let $G$ be a $K_h$-minor-free graph, set $n:=|V(G))|$, and fix $d\in \ds{R}^+$ with $d\geq 1$. Observe that $2d\,g(kn)\geq d\,g(kn)\geq d\,g(1)\geq d$ since $g$ is increasing and $g(1)\geq 1$. Thus, we may assume $1\leq d\leq n$, as otherwise $G$ trivially admits a $(n/d,d)$-fan-partition.
        
        Let $(B_t:t\in V(T))$ be a tree-decomposition of $G$ produced by \cref{GMST}. Thus, $(B_t:t\in V(T))$ has adhesion at most $k$, and each torso is in $\Apex{\scr{G}}{a}$. Fix a root $r\in V(T)$. For each $t\in V(T)\setminus \{r\}$, let $p(t)$ denote the parent of $t$, and let $K_t:=B_t\cap B_{p(t)}$. Note that $|K_t|\leq k$. Set $K_r:=\emptyset$.
        
        Define a weighting $w$ of $T$ via $w_t:=|B_t\setminus K_t|$ for each $t\in V(T)$. Thus, $\sum_{t\in V(T)}w_t=n$. By \cref{treeDeletions}, there is a set $Z$ with $|Z|\leq \ceil{n/d}-1\leq n/d$ such that each connected component of $T-Z$ has weight (under $w$) at most $d$. Let $X:=\bigcup_{z\in Z}K_z$. Note that $|X|\leq k|Z|\leq kn/d$. Let $Z':=Z\cup \{r\}$. Observe that $|Z'|\leq |Z|+1\leq n/d + 1\leq 2n/d$ as $d\leq n$.

        Let $F$ be the forest obtained from $T$ by, for each $z\in Z'\setminus \{r\}$, deleting the edge $p(z)z$. See \cref{FigBreakup}. Consider any connected component $T'$ of $F$. Note that that $T'$ is a subtree of $T$, the induced root of $T'$ is a vertex $z\in Z'$ (since $r\in Z'$), and that $V(T')\cap Z'=\{z\}$. For each $z\in Z'$, let $T_z$ be the component of $F$ containing $z$. Let $G_z:=G[\bigcup_{t\in V(T_z)}B_t\setminus K_t]$, and let $n_z:=|V(G_z)|$. Observe that $V(T)=\bigcup_{z\in Z'}V(T_z)$, and that $V(T_z)$ is disjoint from $V(T_{z'})$ whenever $z,z'\in Z'$ are distinct. Thus, $\bigcup_{z\in Z'}V(G_z)=V(G)$, and $V(G_z)$ is disjoint from $V(G_{z'})$ whenever $z,z'\in Z'$ are distinct. Therefore, $\sum_{z\in Z'}n_z=n$.
        \begin{figure}[!ht]
        \begin{center}
        \begin{minipage}{0.85\textwidth}
        \centering
        \includegraphics[width=\textwidth]{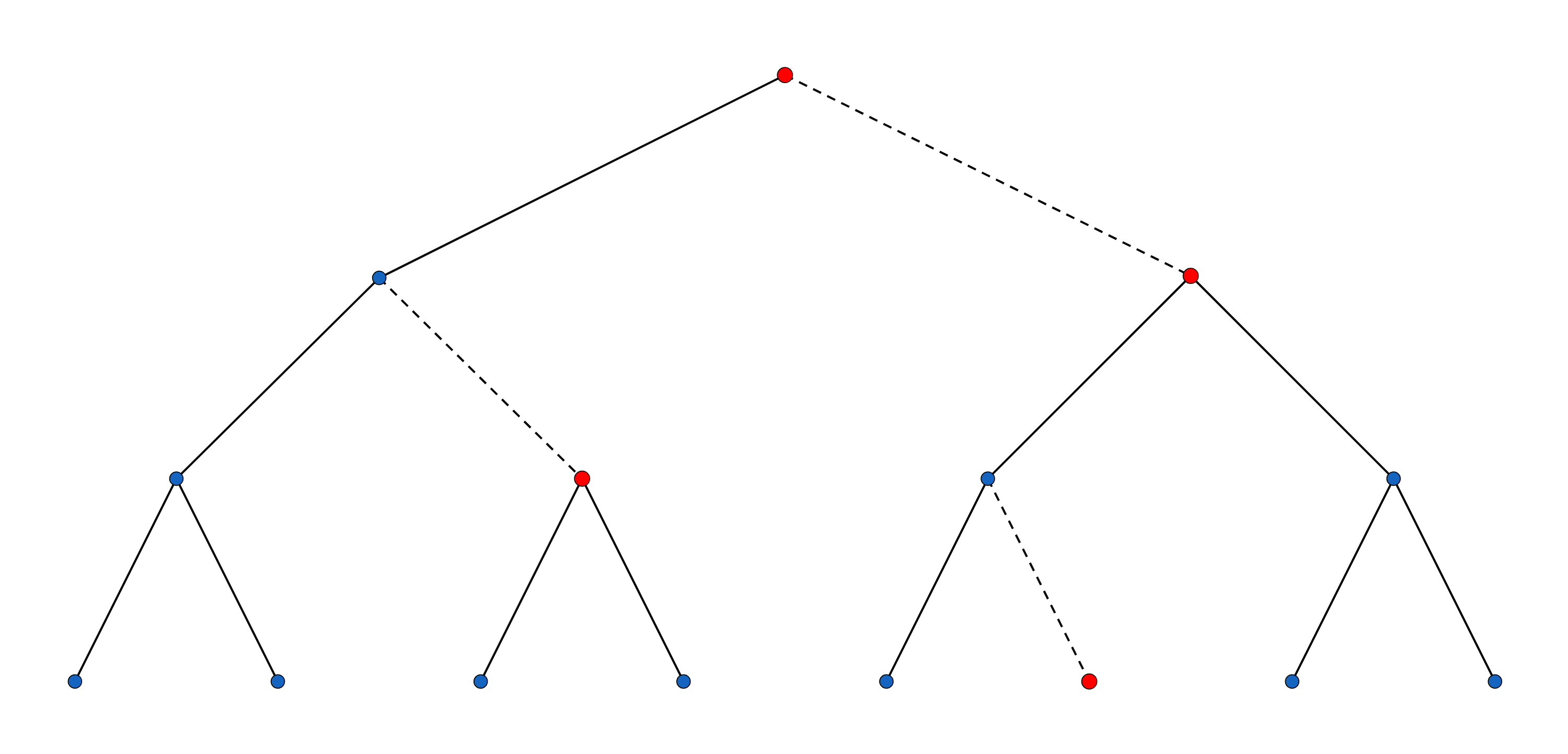}
        \caption{\label{FigBreakup}A diagram of the tree $T$, with the vertices in $Z'$ coloured red and the edges removed to make $F$ represented as dashed lines.}
        \end{minipage}
        \end{center}
        \end{figure}

        As noted above, $\widehat{\Apex{\scr{G}}{a}}=\Apex{\scr{G}}{a}$ admits $(f/d+a,dg)$-fan-partitions. We seek to apply \cref{starDecomp} by showing that, for each $z\in Z'$, $G_z$ has a star-decomposition $(B'_s:s\in V(S))$ of adhesion at most $k$ for some star $S$ with centre $z$ such that $\UWT{G_z}{B'_z}\in \Apex{\scr{G}}{a}$ and $|B'_s\setminus B'_z|\leq d$ for each $s\in V(S)\setminus \{z\}$.

        For each $z\in Z'$, let $\scr{T}_z$ denote the set of connected components of $T_z-z$. Let $S$ be the star with centre $z$ and leaves $\scr{T}_z$. Let $B'_z:=B_z\setminus K_z$, and for each $Q\in \scr{T}_z$, let $B'_Q:=\bigcup_{t\in V(Q)}B_t \setminus K_z$. It follows that $(B'_s:s\in V(S))$ is a star-decomposition of $G_z$ of adhesion at most $k$. By definition of $(B_t:t\in V(T))$, $\UWT{G}{B_z}\in \Apex{\scr{G}}{a}$. As $\UWT{G_z}{B_z'}\subseteq \UWT{G}{B_z}$ and since $\Apex{\scr{G}}{a}$ is monotone, this gives $\UWT{G_z}{B_z'}\in \Apex{\scr{G}}{a}$. By definition of $Z$ and $w$, for each $Q\in \scr{T}_z$, $|B'_Q\setminus B'_z|\leq d$.

        Thus, by \cref{starDecomp}, for each $z\in Z'$, there exists $X_z\subseteq V(G_z)$ with $|X_z|\leq f(kn_z)/d+a$ such that $G_z-X_z$ admits a path-partition of width at most $\max(2d\,g(kn_z),d)=2d\,g(kn_z)$. Here, we use that $g$ is increasing and that $g(1)\geq 1$. 
        
        Let $X':=X\cup \bigcup_{z\in Z'}X_z$. As $f$ is superadditive
        $$\sum_{z\in Z'}|X_z|\leq \sum_{z\in Z'}(f(kn_z)/d+a)\leq f(k\sum_{z\in Z'}n_z)/d+a|Z'|.$$
        Recalling that $\sum_{z\in Z'}n_z=n$ and $|Z'|\leq 2n/d$, this gives
        $$\sum_{z\in Z'}|X_z| \leq (f(kn)+2an)/d.$$
        Recalling $|X|\leq kn/d$, this gives
        $$|X'|\leq (f(kn)+(2a+k)n)/d.$$
        Now, notice that for each connected component $C$ of $G-X'$, $C$ is contained in $G_z-X_z$ for some $z\in Z'$. Thus, $C$ admits a path-partition of width at most $2d\,g(kn_z)\leq 2d\,g(kn)$ as $g$ is increasing. As this is true for each connected component of $G-X'$, $G-X'$ itself admits a path-partition of width at most $2d\,g(kn)$. This gives the desired result.
    \end{proof}

    \begin{theorem}
        \label{flexResult}
        $K_h$-minor-free graphs admit $(f,g)$-flexible fan-partitions, where $f\in O_h(n\log(n))$ and $g\in O_h(\log^3(n))$.
    \end{theorem}

    \begin{proof}
        Let $f,g$ be from \cref{rtwFlex}. So $f\in O(bn\log(n))$ and $g\in O(\log^3(n))$. We may take $f$ to be superadditive, and $g$ to be increasing and to satisfy $g(1)\geq 1$. Thus, \cref{mainLem} is applicable. Finally, observe that $n\mapsto f(kn)+(2\max(h-5,0)+k)n \in O_h(n\log(n))$ and $n\mapsto 2g(kn) \in O_h(\log^3(n))$, where $b,k$ are from \cref{GMST} and depend only on $h$.
    \end{proof}

    \cref{main} follows as a direct corollary of \cref{flexResult,unvToInd} by, for each $K_t$-minor free graph $G$ with $n:=|V(G)|$, fixing $d:=\sqrt{n}/\log(n)$ in the definition of flexible fan-partitions.

    \paragraph{Acknowledgements.} The author thanks David Wood for his supervision and suggestions to improve this paper.

    \bibliography{Ref.bib}

\begin{thebibliography}{10}
\providecommand{\natexlab}[1]{#1}
\providecommand{\url}[1]{\texttt{#1}}
\expandafter\ifx\csname urlstyle\endcsname\relax
  \providecommand{\doi}[1]{doi: #1}\else
  \providecommand{\doi}{doi: \begingroup \urlstyle{rm}\Url}\fi

\bibitem[Alon et~al.(1990)Alon, Seymour, and Thomas]{Alon1990}
Noga Alon, Paul Seymour, and Robin Thomas.
\newblock A separator theorem for nonplanar graphs.
\newblock \emph{J. Amer. Math. Soc.}, 3\penalty0 (4):\penalty0 801--808, 1990.
\newblock \doi{10.2307/1990903}.

\bibitem[Distel et~al.(2022)Distel, Hickingbotham, Huynh, and Wood]{Distel2022Surfaces}
Marc Distel, Robert Hickingbotham, Tony Huynh, and David~R. Wood.
\newblock Improved product structure for graphs on surfaces.
\newblock \emph{Discrete Math. Theor. Comput. Sci.}, 24\penalty0 (2):\penalty0 Paper No. 6, 2022.

\bibitem[Distel et~al.(2024)Distel, Dujmovi\'c, Eppstein, Hickingbotham, Joret, Micek, Morin, Seweryn, and Wood]{Distel2024}
Marc Distel, Vida Dujmovi\'c, David Eppstein, Robert Hickingbotham, Gwena\"el Joret, Piotr Micek, Pat Morin, Micha\l~T. Seweryn, and David~R. Wood.
\newblock Product structure extension of the {A}lon--{S}eymour--{T}homas {T}heorem.
\newblock \emph{SIAM J. Discrete Math.}, 38\penalty0 (3):\penalty0 2095--2107, 2024.
\newblock \doi{10.1137/23M1591773}.

\bibitem[Dujmovi\'{c} et~al.(2020)Dujmovi\'{c}, Joret, Micek, Morin, Ueckerdt, and Wood]{rtwltwMCC}
Vida Dujmovi\'{c}, Gwena\"{e}l Joret, Piotr Micek, Pat Morin, Torsten Ueckerdt, and David~R. Wood.
\newblock Planar graphs have bounded queue-number.
\newblock \emph{J. ACM}, 67\penalty0 (4):\penalty0 Art. 22, 2020.
\newblock \doi{10.1145/3385731}.

\bibitem[Dujmović et~al.(2024)Dujmović, Joret, Micek, Morin, and Wood]{Dujmovic2024}
Vida Dujmović, Gwenaël Joret, Piotr Micek, Pat Morin, and David~R. Wood.
\newblock Planar graphs in blowups of fans, 2024.
\newblock arXiv:2407.05936.

\bibitem[Dvořák and Wood(2023)]{Dvorak2023}
Zdeněk Dvořák and David~R. Wood.
\newblock Product structure of graph classes with strongly sublinear separators, 2023.
\newblock arXiv:2208.10074.

\bibitem[Illingworth et~al.(2022)Illingworth, Scott, and Wood]{Illingworth2022}
Freddie Illingworth, Alex Scott, and David~R. Wood.
\newblock Product structure of graphs with an excluded minor, 2022.
\newblock arXiv:2104.06627. To appear in \emph{Trans. Amer. Math. Soc.}

\bibitem[Lipton and Tarjan(1979)]{Lipton1979}
Richard~J. Lipton and Robert~E. Tarjan.
\newblock A separator theorem for planar graphs.
\newblock \emph{SIAM J. Appl. Math.}, 36\penalty0 (2):\penalty0 177--189, 1979.
\newblock \doi{10.1137/0136016}.

\bibitem[Robertson and Seymour(2003)]{Robertson2003}
N.~Robertson and P.~D. Seymour.
\newblock Graph minors. {XVI}. {E}xcluding a non-planar graph.
\newblock \emph{J. Combin. Theory Ser. B}, 89\penalty0 (1):\penalty0 43--76, 2003.
\newblock \doi{10.1016/S0095-8956(03)00042-X}.

\bibitem[Ueckerdt et~al.(2022)Ueckerdt, Wood, and Yi]{Ueckerdt2022}
Torsten Ueckerdt, David~R. Wood, and Wendy Yi.
\newblock An improved planar graph product structure theorem.
\newblock \emph{Electron. J. Combin.}, 29\penalty0 (2):\penalty0 Paper No. 2.51, 2022.
\newblock \doi{10.37236/10614}.

\end{thebibliography}
\end{document}